\documentclass[twoside]{article}
\usepackage{amsmath,amssymb, amsfonts, graphicx, url,fancyhdr,hyperref}

\newtheorem{theorem}{Theorem}[section]
\newtheorem{lemma}[theorem]{Lemma}

\newtheorem{definition}[theorem]{Definition}

\newtheorem{remark}[theorem]{Remark}
\newtheorem{proof}[theorem]{Proof}

\numberwithin{equation}{section}

\DeclareMathOperator{\acot}{arccot}
\newcommand{\dint}{\; \mathrm{d}}

\title{Integrals of functions containing parameters}

\author{Robert M. Corless,\\
ORCCA and Applied Mathematics, University of Western Ontario\\
\email{rcorless@uwo.ca}}

\author{David J. Jeffrey,\\
ORCCA and Applied Mathematics, University of Western Ontario\\
\email{djeffrey@uwo.ca}}

\author{David R. Stoutemyer,\\
University of Hawaii,\\
\email{dstout@hawaii.edu}}

\newcommand{\bp}{\mathbf{p}}

\begin{document}

\begin{center}
  \textbf{{\Large Integrals of functions containing parameters}}
\end{center}

\noindent Robert M. Corless, David J. Jeffrey and David R. Stoutemyer

\section{Introduction}\label{sec:Introduction}

\begin{flushright}
``\textsl{There is always a well--known solution to every human problem\\
-- neat, plausible, and wrong}.''\\
H. L. Mencken~\cite{MenckenAfflatus}\footnote{The truth of this statement is reinforced by the fact that it is often misquoted.}
\par
\end{flushright}

\begin{center}
  \href{https://www.gocomics.com/nonsequitur/2016/01/20}{The \textsl{Non Sequitur} cartoon for 20 January 2016}
\end{center}

Calculus students are taught that an indefinite integral is defined only
up to an additive constant, and as a consequence generations of students have assiduously added ``\ $+C$\ " to their calculus homework.
Although ubiquitous, these constants rarely garner much attention, and typically loiter without intent
around the ends of equations, feeling neglected.
There is, however, useful work they can do; work which is particularly relevant in the contexts of integral tables and computer algebra systems. We begin, therefore, with a discussion of the context, before returning to coax the constants out of the shadows and assign them their tasks.

Tables of integrals are inescapable components of calculus textbooks~\cite{ Essex, Stewart},
and there are
well known reference books that publish voluminous collections~\cite{Abramowitz1964,Gradsh,PrudnikovBM,CRCtables}.
A modern alternative to integral tables is provided by computer algebra systems (CAS),
which are
readily 
available on computing platforms ranging from phones to supercomputers.
These systems evaluate integrals using a mixture of integral tables and algebraic algorithms.
A feature shared by tables and computer systems is the fact that the formulae usually contain parameters.
No one would want a table of integrals that contained separate entries for $x$, $x^2$ and $x^{42}$, rather than one entry for $x^n$, and many tables include additional parameters
for user convenience; for example, there will be entries for integrals containing
$\sin ax$, rather than the sufficient, but less convenient, $\sin x$.

Although parameters add greatly to the scope and convenience of integral tables,
there can be difficulties and drawbacks occasioned by their use.
We shall use the word \textsl{specialization} to describe the action of substituting specific values (usually numerical, but not necessarily) into a formula.
The  \textsl{specialization problem} is a label for a cluster of
issues 
associated with formulae and their specialization,
the
difficulties 
ranging from inelegant results to invalid ones.
For example, in \cite{JeffreyRich2010} an example is given in which the evaluation
of an integral by specializing a general
formula misses a particular case for which a more elegant expression is possible.
The focus here, however, is on situations in which specialization leads to invalid or incorrect results.
To illustrate the problems, consider
an example drawn from a typical collection~\cite[ch8, p346, (5)]{Timofeev1948}:
\begin{equation}\label{eq:AMultiparameterGenericAntideriv}
I_1=\int\left(\alpha^{\sigma z}-\alpha^{\lambda z}\right)^2\,dz =
\frac1{2\ln\alpha}\left(\frac{\alpha^{2\lambda z}}{\lambda}+\dfrac{\alpha^{2\sigma z}}{\sigma}-\dfrac{4\alpha^{(\lambda+\sigma)z}}{\lambda+\sigma}\right)\ .
\end{equation}
Expressions equivalent to this are returned by Maple, Mathematica and many other systems, such as the Matlab symbolic toolbox.

Before we proceed, we acknowledge that some readers may question whether anyone at all competent would write the integral this way: surely there are better ways?  Why not transform $\alpha^{\sigma z}$ into $\exp( p z)$, where $p=\sigma\ln\alpha$,
and thus reduce the number of parameters?
Or scale the variable of integration to absorb, say, the $\lambda$ ?
Such actions are possible for people who are free to recast problems in convenient ways,
for example, if \eqref{eq:AMultiparameterGenericAntideriv} were an
examination question devoid of context.
CAS, in contrast, are obliged to deal with expressions as they are presented,
either by users or by other components within the system itself;
and \textsl{in the general case} some of these ``obvious'' simplifications and transformations are surprisingly difficult to discover automatically.  Humans are still superior at simplification, we believe.

Returning to the answer as returned by the CAS, it is easy to see that the specialization $\sigma=0$ leaves
the left side of \eqref{eq:AMultiparameterGenericAntideriv}, the integrand, well defined,
but the expression for its integral on the right-hand side is no longer defined.
If we pursue this further, we see that there are multiple specializations for which
\eqref{eq:AMultiparameterGenericAntideriv} fails, \textsl{viz.} $\alpha=0$, $\alpha=1$, $\lambda=0$, $\sigma=0$, $\lambda=-\sigma$, and combinations of these.
The question of how or whether to inform computer users of these special cases has been discussed
in the CAS literature many times~\cite{CorlessJeffrey:Simple}.

This brings us to the second theme of this discussion: comprehensive and generic results.
A comprehensive result lists explicit expressions for each set of special parameter values, while a generic result
is correct for `most' or `almost all' values of the parameters.
Let us consider how a comprehensive result for \eqref{eq:AMultiparameterGenericAntideriv} would look.

\begin{equation}\label{eq:ResultOfEg1ForSeekContinuityFalse-1}
I_1=\begin{cases}
\displaystyle \frac1{2\lambda\ln \alpha}\left(\alpha^{2\lambda z}-
   \alpha^{-2\lambda z}-4z\lambda\ln\alpha\right)\ ,
&\left[\begin{tabular}{l}
    $\lambda+\sigma=0$, $\sigma\neq0$\ ,\\
    $\alpha\neq 0$\ ,\ $\alpha\neq1$\ \ ;
  \end{tabular}
\right.\\[15pt]
\displaystyle z+\frac1{2\lambda\ln\alpha}\left(\alpha^{\lambda z}(\alpha^{\lambda z}-4)\right)\ ,
&\left[\begin{tabular}{l}
          $\sigma=0$, $\lambda\neq0$\ ,\\
          $\alpha\neq 0$\ ,\ $\alpha\neq1$\ \ ;
          \end{tabular}\right.\\[15pt]
\displaystyle z+\frac1{2\sigma\ln\alpha}\left(\alpha^{\sigma z}(\alpha^{\sigma z}-4)\right)\ ,
& \left[\begin{tabular}{l}
          $\lambda=0$\ ,\ $\sigma\neq0$\\
          $\alpha\neq 0$\ ,\ $\alpha\neq1$\ ;
        \end{tabular}
\right.\\[15pt]
0 \ ,
&\left[\begin{tabular}{l}
          $\alpha=1 $\ ;
        \end{tabular}
\right.\\[5pt]
0 \ ,
& \left[\begin{tabular}{l}
          $\alpha=\lambda=\sigma=0$\ ;
        \end{tabular}
\right.\\[5pt]
\mathrm{ComplexInfinity}\ ,
& \left[\begin{tabular}{l}
    $\alpha=0$\ , \\
    $\Re(\lambda z)\,\Re(\sigma z)<0\ ;$
  \end{tabular}
\right.\\[18pt]
\mathrm{Indeterminate}\ ,
&\left[ \begin{tabular}{l}
    $\alpha=0$\ , \\
    $\Re(\lambda z)\,\Re(\sigma z)\geq 0$\ ;
  \end{tabular}\right.
\\[15pt]
\displaystyle \frac1{2\ln\alpha}\left(\frac{\alpha^{2\lambda z}}{\lambda}+\frac{\alpha^{2\sigma z}}{\sigma}-\frac{4\alpha^{(\lambda+\sigma)z}}{\lambda+\sigma}\right)\ ,
& \textrm{otherwise (generic case)}
\\
\end{cases}
\end{equation}
Conditions are here shown as in printed tables; otherwise they could be presented using the logical $\vee$ and $\wedge$ operators.

To generalize, we denote a function depending on parameters by $f(z; \bp)$, with $z$ being the main argument,
here the integration variable, and $\bp$ a list of parameters. The definition is then:
\begin{definition}
A \textsl{comprehensive antiderivative} of a parametric function $f(z;\bp)$ is a piecewise function $F(z;\bp)$ containing explicit consequents\footnote{\textsl{consequent}: following as a result or effect; the second part of a conditional proposition, dependent on the antecedent.} for each special case of the parameters.
\end{definition}

Designers of computer algebra systems are reluctant to return comprehensive expressions
by default, because they can quickly lead to unmanageable computations,
and as well many users might regard them as \textsl{too much information}.
Instead, tables and CAS commonly adopt the approach of identifying a \textsl{generic} case, which is then the only
expression given; in the case of CAS, the generic case is typically returned without explicitly showing the conditions on the parameters. In the case of tables, any special-case values would be used to simplify the integrand and then the resulting integrand and its antiderivative would be displayed as a separate entry
somewhere else in the table.
\begin{definition}
A \textsl{generic antiderivative} is one expression chosen from a comprehensive antiderivative that is valid for
the widest class of constraints.
\end{definition}

\begin{remark}
The above definition is an informal one, since the choice of which result to designate as generic may
include personal taste.
\end{remark}

\section{Continuity in parameters}
We now come to the third theme of the discussion: the treatment of removable discontinuities.
Consider the improper integral
\begin{equation}\label{eq:improperint}
  \int_{0}^{1} \ln \left( x(1-x)\right)\,dx =\bigg[ x\ln\left(x(1-x)\right) -2x-\ln(1-x)\bigg]_0^1\ .
\end{equation}
The expression for the integral contains a removable discontinuity at each end, and a computer system
(and we hope students) would automatically switch to limit calculations to obtain the answer.
In this section,  by extending
the handling of removable discontinuities to constants of integration, we introduce a new idea for handling the specialization problem in
integration\footnote{The specialization problem is not confined to integration.
Any formula which uses parameters to cover multiple cases is likely to have some specialization problems.
The ideas presented here, however, apply specifically to integration.}.
The idea is new in the sense that we do not know of any published discussion, but it originated with William Kahan~\cite{Kahan59} and was circulated informally possibly as early as $1959$.

The example \eqref{eq:ResultOfEg1ForSeekContinuityFalse-1} dramatically illustrates the potential size of comprehensive antiderivatives, but is, unsurprisingly, too cumbersome for explaining ideas.
We turn to simpler examples.
We begin with the comprehensive antiderivative known to all students of
calculus\footnote{Most textbooks use $x^n$, but we wish to emphasize continuity and use $\alpha$ instead of $n$.
We also use $z$ and $\alpha$ because we are thinking in the complex plane, and we prefer
$\ln z$ instead of $\ln|x|$ for the same reason.}:
\begin{equation}\label{eq:intresipz}
  \int z^\alpha \dint z = \begin{cases} \ln z\ , & \mbox{if } \alpha= -1\ , \\
                     \dfrac{z^{\alpha+1}}{\alpha+1}\ , & \mbox{otherwise (generic case)}.
                     \end{cases}
\end{equation}
Substituting $\alpha=-1$ into the generic case gives $1/0$ and not $\ln z$. Often when a substitution fails, a limit will succeed, so we try the limit as $\alpha\to -1$. Disappointingly, this also fails, but we can examine \textsl{how} the limit fails by expanding the generic case as a series about the pole at $\alpha=-1$, that is, treating $\alpha+1=\varepsilon$ as a small quantity.
\begin{equation}\label{eq:loglimit}
      \frac{z^{\alpha+1}}{\alpha+1}=\frac{e^{\varepsilon\ln z}}{\varepsilon}=\frac{1+\varepsilon\ln z +O\left(\varepsilon^2\right)}{\varepsilon}
      \ .\end{equation}
If we can remove the leading term of the series, namely $1/\varepsilon$, then the next term gives us
the desired $\ln z$.
But an integral needs a
constant\footnote{The lack of a constant in \eqref{eq:intresipz} betrays our CAS allegiance: it is a rare CAS that adds a constant, because the user can easily add one (and name the constant) by typing, for instance, ``\texttt{int(f,x)+K}".}! So, an equally correct integral is
\[ \int z^\alpha \dint z = \dfrac{z^{\alpha+1}}{\alpha+1}-\frac{1}{\alpha+1}\ ,\]
and now the limit
as $\alpha\to -1$ is precisely $\ln z$.
Thus the comprehensive antiderivative,
\begin{equation}\label{eq:xtoncomp}
  \int z^\alpha \dint z =\begin{cases}
                      \ln z\ , & \mbox{if } \alpha= -1\ , \\[5pt]
                      \dfrac{z^{\alpha+1}-1}{\alpha+1}\ , & \mbox{otherwise}.
                    \end{cases}
\end{equation}
is continuous with respect to $\alpha$, and the generic antiderivative now contains the exceptional case as a removable discontinuity~\cite{Kahan59}.

\begin{figure}
\caption{A parametrically discontinuous integral. The real part of each consequent of the comprehensive integral \eqref{eq:DisconInt1OnzSqrtzSqMinusaSq} is plotted. The surface shows the generic expression plotted as a function of the integration variable $z$ and the parameter $\alpha$. The surface curves up to become singular all along the line $\alpha=0$. The detached curve hovering below the surface in the plane $\alpha=0$ is the special case integral.}
\label{fig:arccotdis}
\centering
\includegraphics[scale=0.5]{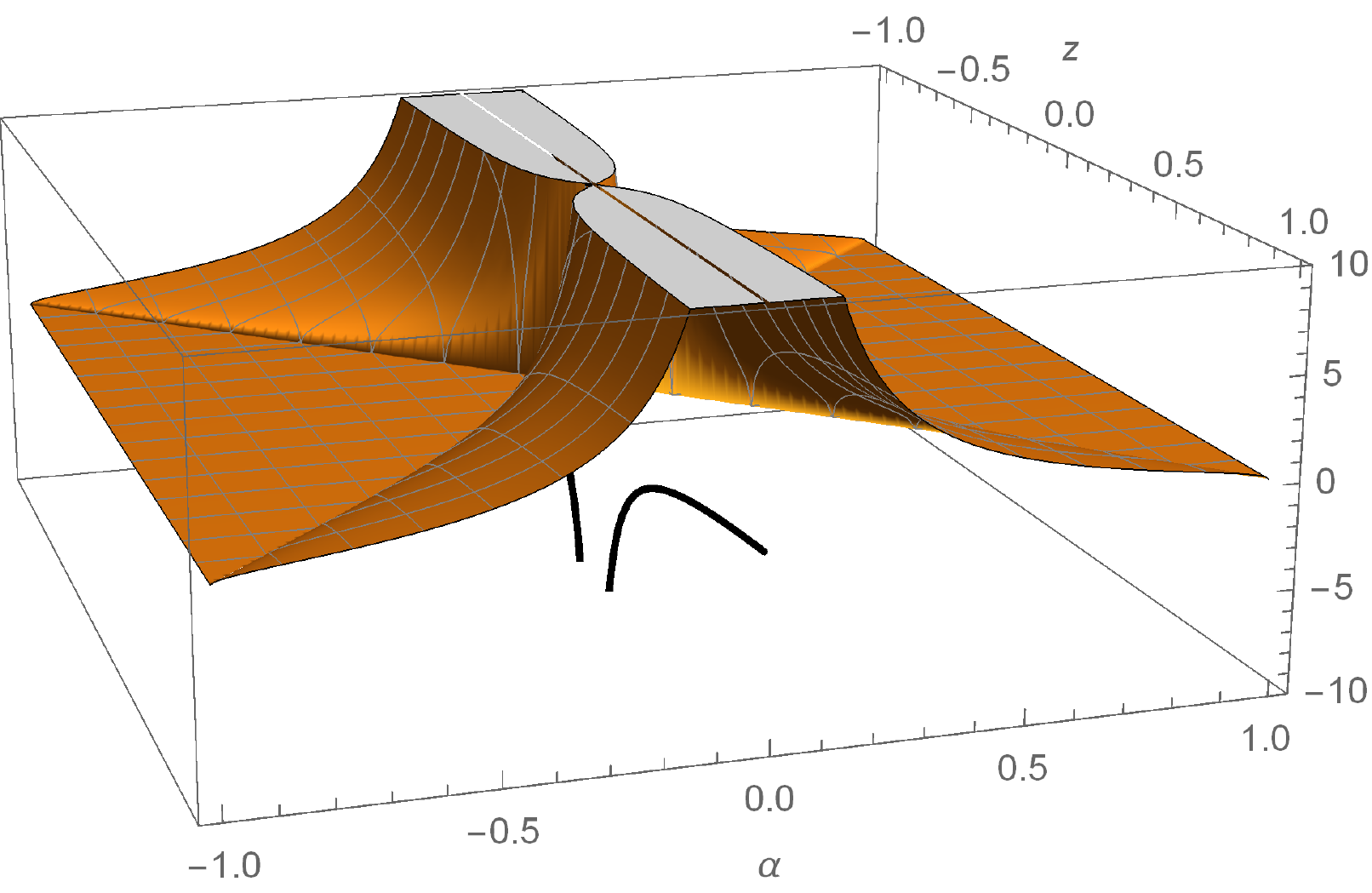}
\end{figure}

\begin{definition}
Let a function $F(z;\bp)$ be an indefinite integral of an integrand $f(z;\bp)$. That is,
\[ F(z;\bp)=\int f(z;\bp)\dint z\ .
\]
If a point $\bp_c$ in parameter space exists at which $F(z;\bp)$ is \textsl{discontinuous} with respect to one or more members of $\bp$, 
and if
a function $C(\bp)$, which serves as a constant of integration with respect to $z$, 
has the property
that $F(z;\bp)+C(\bp)$ has only a removable discontinuity and thus can be made \textsl{continuous} with respect to $\bp$ at $\bp_c$, then $C(\bp)$ is called a Kahanian constant of integration\footnote{Since it is a function of $\bp$, one can question whether it should be called a constant. It is constant with respect to $z$, and this seems a useful extension of calculus terminology.}.
\end{definition}

\begin{remark}
The definition does not guarantee the existence of $C(\bp)$.
If, for some value of $\bp$, the integral does not exist, then there will be no Kahanian.
\end{remark}

A second example shows the effect of Kahanian constants graphically.
To begin with, consider the comprehensive antiderivative
\begin{equation}\label{eq:DisconInt1OnzSqrtzSqMinusaSq}
\int \frac{dz}{z\sqrt{z^2-\alpha^2}}=\;\begin{cases}
-\dfrac{1}{\sqrt{z^2}}\ , & \alpha=0\ ,\\[10pt]
\dfrac1\alpha\;{\acot\dfrac{\alpha}{\sqrt{z^2-\alpha^2}}}\ , & \mbox{generic}\ .
\end{cases}
\end{equation}
We write $\sqrt{z^2}$ rather than $|z|$ so that the expression is valid for non-real $z$.
Figure~\ref{fig:arccotdis} shows this antiderivative as a three-dimensional plot
treating both $z$ and $\alpha$ as real variables. The generic expression is then
the surface shown in the plot; it becomes singular along the line $\alpha=0$
(as $\alpha^{-2}$).
The special case $\alpha=0$ is shown as a detached curve confined to the plane $\alpha=0$. It can be seen hovering forlornly underneath the surface of the generic integral, dreaming of
gaining an invitation to the party.
For $|z|<|\alpha|$, the values of the integral are non-real, but only the real part is plotted, because that displays the properties of interest\footnote{The definition of arccot varies between computer systems and amongst references, and even between different printings of the same reference work~\cite{corless2000according}. The plots shown here were made with Mathematica, and other systems such as Maple and Matlab may create different plots.}.

\begin{figure}
\caption{A parametrically continuous integral. The real part of each consequent of the comprehensive integral \eqref{eq:CompInt1OnzSqrtzSqMinusaSq} is plotted. The surface shows the generic expression plotted as a function of the integration variable $z$ and the parameter $\alpha$. The surface curves down, and is singular only at $z=\alpha=0$. The special case curve lies happily within the surface in the plane $\alpha=0$.}
\label{fig:arccotcon}
\centering
\includegraphics[scale=0.4
]{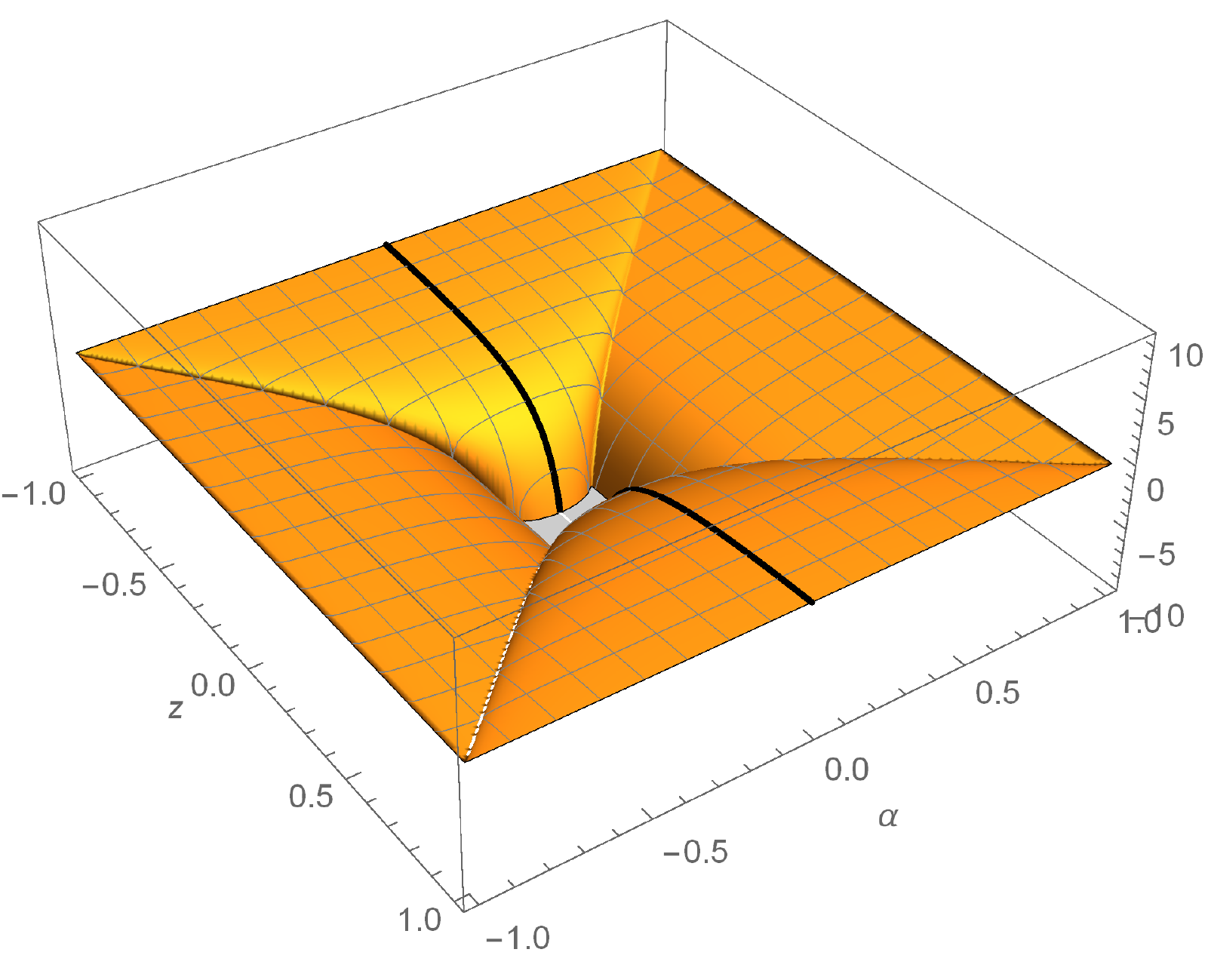}
\end{figure}

To achieve continuity, we now add Kahanian constants to each consequent of the comprehensive antiderivative.
We show the new constants in bold print.
\begin{equation} \label{eq:CompInt1OnzSqrtzSqMinusaSq}
\int \frac{dz}{z\sqrt{z^2-\alpha^2}}=\;\begin{cases}
-\dfrac{1}{\sqrt{z^2}}+\boldsymbol{1}\ , \\[10pt]
{\dfrac1\alpha\;{\acot\dfrac{\alpha}{\sqrt{z^2-\alpha^2}}}}
-\boldsymbol{\dfrac1\alpha\;{\acot\dfrac{\alpha}{\sqrt{1-\alpha^2}}}}\ .
\end{cases}
\end{equation}

Figure \ref{fig:arccotcon} shows the real part of the new expressions.
The extra term in the generic (the lower) consequent in (\ref{eq:CompInt1OnzSqrtzSqMinusaSq}) subtractively cancels
a parametric pole asymptotic to $1/\alpha^{2}$ as $\alpha\rightarrow0$,
which can be seen in the
generic consequent of result (\ref{eq:DisconInt1OnzSqrtzSqMinusaSq}).
This makes the parametrically continuous expression now approach the
same values from both sides of $\alpha=0$, converting that parametric
pole to an indeterminate slit in an otherwise parametrically continuous
surface---a removable singularity. Moreover, the extra term $1$
in the $\alpha=0$ consequent of (\ref{eq:CompInt1OnzSqrtzSqMinusaSq})
raises the space curve exactly the right amount to make it contiguous
with the surface on both sides, thus removing that removable singularity.

\section{\!Advantages of parametrically continuous integrals}
\label{sec:advantages}
We have defined the specialization problem as the failure of a generic formula when
particular values are substituted for parameters, with the generic integrals above being examples.
One way to avoid specialization problems completely would be to specify all parameters in advance, that is, delay starting a calculation until the parameters are known.
This, however, negates the very power and generality that algebra extends to us.
As well, it is not always the case that one can know parameters in advance; for example, a parameter might depend on the outcome of an intermediate computation.

The specialization problem can be looked at another way, a way perhaps more suitable for computer algebra. There are variations possible in the order in which operations are applied.
For example, in Maple syntax, where as usual operations take precedence from the inside out, it is the difference between
\begin{verbatim}
                    subs( n=-1, int( x^n, x ) ) ;
                    int( subs( n=-1, x^n ), x ) .
\end{verbatim}
The first gives a division by zero error; the second gives $\ln x$.
The challenge, then, is to retain the generality of algebra, while finding ways to react to exceptional values of parameters.

\subsection{Definite integration}
There are at least two ways to evaluate a special case of a known parametric definite integral, especially in the computer-algebra setting.
The first way is the typical human way: by working on each limit separately.
Thus to evaluate
\[\int_{a}^{b} f(x,\bp)\dint x\ ,\]
one first finds an indefinite integral $F(x,\bp)$, then makes any substitutions $F(x,\bp_c)$, then evaluates and simplifies $F(b,\bp_c)-F(a,\bp_c)$. For this approach, the Kahanian term is important, because it allows a substitution to be performed correctly, possibly as a limit $\lim_{\bp\to\bp_c} F(x,\bp)$.

For the second way, let us explicitly notate the presence of a Kahanian term, so that the indefinite integral is $F(x,\bp)+C(\bp)$, where now $F$ is any function that satisfies $F^\prime=f$ and $C(\bp)$ is the Kahanian constant.
We now perform the definite integral \textsl{before} specializing.
\begin{align}
 \int_{a}^{b} f(x,\bp)\dint x &= \left[\ F(b,\bp) + C(\bp)\ \right] -
                                  \left[\ F(a,\bp) + C(\bp)\ \right] \nonumber \\
                              &= F(b,\bp)-F(a,\bp)\ .
\end{align}
Now we must evaluate this expression as $\bp\to\bp_c$. If we simplify each term separately, the calculation may fail, but \textsl{keeping the terms together}, the calculation succeeds.

For example,
\[\lim_{n\to -1} \int_a^b x^n\dint x = \lim_{n\to -1} \left(\frac{b^{n+1}}{n+1}-\frac{a^{n+1}}{n+1}\right)=\ln b-\ln a\ ,\]
using the methods of \eqref{eq:loglimit}.
Again we point out that separately the limits of each term need not exist.
Computer systems, like people, can sometimes succeed when asked one way and not another.
With the Kahanian form, either approach succeeds.

\subsection{Resonance}
We now give an example of Kahanian terms used in the solution of a differential equation.
A standard topic in physics and engineering is resonance.
The equation of a forced, frictionless, harmonic oscillator
is
\begin{equation}\label{eq:ForcedOscillator}
  \frac{d^2 x(t)}{dt^2}+k^2 x(t) = \cos\omega t\ ,
\end{equation}
It has the generic general solution
\begin{equation}\label{eq:ForcesOscGeneric}
   x(t) = C_1 \cos kt + C_2 \sin kt + \frac{\cos\omega t}{k^2-\omega^2}\ ,
\end{equation}
where $C_1$ and $C_2$ depend on initial conditions, and the last term is the particular integral.
The phenomenon of resonance
occurs when $\omega^2=k^2$ and the particular integral becomes invalid.
We shall now derive a particular integral containing Kahanian terms that enable the particular integral to have a valid limit in the resonant case.
We shall use the method of variation of parameters~\cite{BenderOrszag, JeffreysJeffreys}.
We start from the solutions to the homogeneous equation: $x_1=\cos kt$ and $x_2=\sin kt$.
Then the particular integral is given by $x_p=u_1 x_1 + u_2 x_2$\,, where
\[
 u_1 = -\int \frac{x_2\cos\omega t}{W} \dint t \ , \qquad
 u_2 = \int\frac{x_1\cos\omega t}{W} \dint t \ ,
\]
and $W=x_1x_2^\prime-x_2x_1^\prime$ is the Wronskian.
Evaluating the integrals in the usual way, we obtain
\begin{align}\label{eq:VarParamsKahan}
  u_1 &= \frac{\cos((k-w)t)}{2k(k-w)}+\frac{\cos((k+w)t)}{2k(k+w)}\ , \\
  \label{eq:VarParamsKahan2}
  u_2 &= \frac{\sin((k-w)t)}{2k(k-w)}+\frac{\sin((k+w)t)}{2k(k+w)}\ .
\end{align}
When the expression for $x_p$ is simplified, we are led to \eqref{eq:ForcesOscGeneric}.
If, however, we change to Kahanian antiderivatives, we obtain
\begin{align}\label{eq:VarParams}
  u_1 &= \frac{\cos((k-w)t)-1}{2k(k-w)}+\frac{\cos((k+w)t)-1}{2k(k+w)}\ ,
\end{align}
and no Kahanian is needed for $u_2$, so it is still given by \eqref{eq:VarParamsKahan2}.
The new particular integral is
\begin{equation}\label{eq:NewOscillator}
  x_p(t) = \frac{\cos\omega t-\cos kt}{k^2-\omega^2}\ ,
\end{equation}
and now the limits $\omega\to \pm k$ give
\begin{equation}\label{eq:Newresonant}
  x_p(t) = \frac{t}{2k} \sin kt\ ,
\end{equation}
showing an oscillation that increases with time---a hallmark of resonant behaviour\footnote{This is an example of a computation that \textsl{ought} to be routine, for a human using familiar trigonometric identities. Instead of thinking back to how you solved it in your first course in differential equations, correctly accounting for resonance, imagine that you are a computer subroutine, having to turn out a good answer by a mechanical algorithm. In that context, the Kahanian approach makes automation easier.}.

Both Mathematica and Maple return \eqref{eq:ForcesOscGeneric} without provisos,
and to obtain
\eqref{eq:Newresonant}, one must substitute $\omega=k$ and rerun the solution.
The lesson from this example is that, because of their black-box automation,
computer algebra systems should implement and exploit comprehensive
results from the most basic operations on up through the most sophisticated.
For example, it would greatly help to have comprehensive limits and
comprehensive series, as well.

\section{Computing a Kahanian antiderivative}\label{sec:How-to-continuous}

The derivation used in \eqref{eq:xtoncomp} appears \textsl{ad hoc}, but a more systematic procedure is possible. Instead of computing an indefinite integral, we calculate a ``semi-definite'' integral.

\begin{definition}
  A \textbf{parametric semidefinite integral} is
one of the form
\begin{equation} \label{def:semidef}
P(z;\bp)=\int_{A}^{z}f(\tau;\bp)\dint \tau,
\end{equation}
where the lower limit $A$ is called the \textbf{anchor point}, and is constant.
\end{definition}
\begin{remark} Some people use the term \emph{indefinite integral} to describe \eqref{def:semidef}, allowing one point (usually the lower limit) to be fixed and the other to vary. We, however, have earlier used the term \emph{indefinite integral} to mean \textsl{any} antiderivative or primitive, as is common parlance where we work.
To avoid confusion and to fix attention on the anchor point, we have introduced the term ``semi-definite'', although we are somewhat indefinite (that is, semi-indefinite) about the hyphen in the term.
\end{remark}

\begin{lemma}
In \eqref{def:semidef}, let $z$ be finite and let $f(\tau;\bp)$ be continuous with respect to $\bp$ and with respect to $\tau$ in the domain of interest containing $z$,  $\tau$ and $A$; $A$ is a fixed finite numeric constant.
Then $P(z;\bp)$
is continuous with respect to all finite values of its parameters
$\bp$, except perhaps for removable singularities.
\end{lemma}

\begin{proof}
This theorem is a consequence of classical theorems about the interchange of limits when functions are uniformly continuous: see for instance~\cite{Apostol57}. As a conceptual alternative, consider the following.
Let $G(z;\bp)$ be a generic antiderivative of $f(z;\bp)$. Then
\begin{equation}\label{eq:SemiDefViaSubstitution}
P(z;\bp)=G(z;\bp)-G(A;\bp)\ ,
\end{equation}
and all discontinuities whose locations depend only on the parameters occur both
in $G(z;\bp)$ and in $G(A;\bp)$. Therefore
they \textsl{cancel} in the semidefinite integral (\ref{eq:SemiDefViaSubstitution}),
leaving at worst an expression that is \textsl{indeterminate} at the
locations of those discontinuities, making the discontinuities removable.
\end{proof}
\begin{remark}
The finiteness of the interval in the hypotheses is necessary.  If the interval of integration is unbounded, then the lemma need not be true.  For example, for real $x$ we have (the signum function is $-1$ for negative arguments, $+1$ for positive arguments, and zero for zero argument)
\begin{equation}
    \int_{-\infty}^x \frac{\sin pt}{t}\,dt = \frac{\pi}{2}\mathrm{signum}(p) + \int_0^{px} \frac{\sin u}{u}\,du
\end{equation}
which has a jump discontinuity at $p=0$ although the integrand $\sin(pt)/t$ is continuous there.  We are indebted to a referee for this example.
\end{remark}

\begin{remark}
The complexity of the Kahanian depends on the choice of anchor point
$A$. We want to avoid values of $A$ that make $G(A;\bp)$
indeterminate or take an infinite magnitude.
The least complex Kahanian is $0$. Therefore it is worth comparing the
complexities of Kahanian constants corresponding to different anchor points,
and preferring any that yield a Kahanian of $0$.
\end{remark}
\subsection{Earlier examples}
The example \eqref{eq:xtoncomp} is obtained from the semidefinite integral
$\int_{1}^{z} \tau^\alpha \dint \tau$.
Our opening example \eqref{eq:AMultiparameterGenericAntideriv} can be modified by
using an anchor point $A=0$ to calculate a Kahanian
constant: we obtain
\begin{align}\label{eq:OpeningWithKahanian}
  \int\left(\alpha^{\sigma z}-\alpha^{\lambda z}\right)^2\,dz =& \nonumber \\
  \boldsymbol{\frac{-(\lambda-\sigma)^2}{2\lambda\sigma(\lambda+\sigma)\ln\alpha}}&+
\frac1{2\ln\alpha}\left(\frac{\alpha^{2\lambda z}}{\lambda}+
\dfrac{\alpha^{2\sigma z}}{\sigma}-\dfrac{4\alpha^{(\lambda+\sigma)z}}{\lambda+\sigma}\right)\ .
\end{align}
It is straightforward to verify that by taking the limits $\alpha\to 0$, $\lambda\to 0$ and
$\lambda\to -\sigma$, each of the special cases in \eqref{eq:ResultOfEg1ForSeekContinuityFalse-1}
is reproduced. That is, the use of a Kahanian makes each of the special cases of the comprehensive integral into a removable discontinuity.  Therefore, the conceptual advance of replacing \emph{evaluation using substitution} by \emph{evaluation using limits}, as discussed previously, can be usefully applied.

\section{Implications}
The danger of exceptional values is as old as algebra.
Early enthusiasts, amazed at algebraic power, often overlooked exceptions.
Eventually, though, experts such as Cauchy and Weierstrass worked to check the enthusiasm,
pointing out that care was needed.
Out of this care, modern analysis was born; out of the enthusiasm, modern algebra.
Hawkins \cite{Hawkins1977weier} writes of Cauchy
\begin{quotation}
``Nevertheless, Cauchy did not accept the particular algebraic foundation
used by Lagrange \ldots{} Cauchy, however, had well-founded doubts
about the automatic general interpretation of symbolic expressions.
He had warned that ``most (algebraic) formulas hold true only under
certain conditions, and for certain values of the quantities they
contain.''
\end{quotation}

The dichotomy between analysis and algebra survives to this day, and can be seen in the present discussion
of integration.
Within computer algebra, integration is based on modern algebraic algorithms, such as the Risch algorithm and its generalizations~\cite{Risch1970, Bronstein2005}, for solving the elementary antidifferentiation problem.
Algebraic algorithms explicitly exempt the constant of integration from consideration.
Indeed, when verifying an integration formula by differentiation, all piecewise constants are given zero derivatives.

\subsection{Computer algebra system design}

In the early days of computer algebra, implementers and users were
equally delighted at the ability of systems to obtain generic results.
Now, however, the low-hanging fruit has been harvested, and implementers
can and should make another pass through their fundamental functionalities
such as integration, the solving of systems of equations and inequalities,
limits and series to make as many results as possible comprehensive.
The robustness of high-level functionality demands it.

\subsection{Mathematical tables}

Computer algebra and the internet make it decreasingly likely that
there will be a completely new printed table of integrals.
There will, however, probably be new editions of existing tables, because there
is something to be learned by scanning a table of closely-related
integrals, as opposed to seeing results one at a time from a computer
algebra system, with no organizing principle.

Printed integral tables would be impractically bulky if every entry
were a piecewise result of the kind in this article. However, many
of the special cases could be listed once in the most appropriate
place, then explicitly cross referenced from more generic cases. New
editions can also make the relevant domains as general as possible
and more explicitly obvious.

Moreover, an \textsl{on-line} version could assemble each piecewise
result as needed. There could even be a computer algebra system involved.
The difference of such a \textsl{mathematical knowledge base} from
a bare computer algebra system is that a user can learn by browsing
through related examples that follow a clear organizing principle.

\subsection{Mathematical practice}

\begin{flushright}
``\textsl{A lot of times, people don't know what they want until
you show it to them}.''\\
--Steve Jobs
\par\end{flushright}

Mathematicians are increasingly frequent users of computer algebra
and other educational or research mathematical software. We suspect
that when users start encountering more comprehensive results, they
will become disappointed in software that does not provide them. As
a side benefit, users might become more careful about not overlooking
special cases or relevant issues such as domain enforcement and continuity
issues with their manually derived results. Perhaps editors and referees
will also pay more attention to such details in articles they are
reviewing.

Meanwhile we hope that this article serves as a warning that when
a computer algebra system returns a generic parametric antiderivative,
the user should ponder the result, to determine whether there exist special cases,
and if so, compute them with separate integrands.

We think that the idea of comprehensive antiderivatives will be welcomed
by many mathematicians,
perhaps after they are exposed to it through using future versions of
computer algebra systems that offer built-in
comprehensive antiderivatives.
We admit, however, that \textsl{parametrically continuous}
antiderivatives will be adopted only slowly, because they are not
always necessary, and they are usually more complicated than the simpler, incorrect answer.
In addition, from a numerical point of view, they can
suffer from catastrophic numerical cancellation, which requires higher precision (or perhaps some numerical analysis experience) to overcome.
Nonetheless, the point remains that when they are necessary, they are crucial to obtaining a correct and complete
result.

\subsection{Implications for mathematics education}

The calculus curriculum is already quite full, and it seems unlikely
that the textbook examples and the expected exercise or test results
could all be comprehensive results. However it does seem worthwhile
to introduce students to the concept and have them do some simple
examples so that those who proceed into mathematical careers (or careers
that use mathematics) are more thorough and careful.
The concept and ideas behind a Kahanian antiderivative also seem worthwhile,
but we do not expect that to be as prevalent.

\subsection{Adoption issues}

We are well aware that an abrupt transition to Kahanian parameterized
antiderivatives is impossible. One of us implemented the non-piecewise
generic portion of Kahan's example (\ref{eq:xtoncomp}) in an
early version of a computer algebra program. The complaints were immediate,
numerous and strong. Even many mathematics teachers said that it was
incorrect even though it differed only by a constant from the traditional
antiderivative. Perhaps the space-saving omission of a generic integration
constant $C$ from most integral tables led some mathematics teachers
to forget what they taught, or perhaps they were simply concerned
that it would alarm, and hence intimidate, many of their students. Another
of us taught the concept to their beginning calculus students. They
hated it.

\subsection{The pedagogical value of ``lies to children''}
The Wikipedia entry on \textsl{Lie-to-children} quotes \cite{JeffreyCorlessLies}, ``The pedagogical point is to avoid unnecessary burdens on the student’s first encounter with the concept.'' Asking students, on their first encounter with antiderivatives, to worry about the continuity of their answer, in addition to other worries, might seem unreasonable.  Asking computer algebra systems to cater to the needs of first-year students as well as the needs of people who solve differential equations with parameters might also seem unreasonable, without some sort of switch to ``expert mode'', say.
On the other hand, if only experts learn to be careful with parametric continuity, then both education and software have done a disservice to their audiences.

\subsection{Closing remarks}
In one sense this paper merely offers a minor technical correction
to the current practice of computing indefinite integrals. However,
the total impact of this minor correction is potentially large because the
current practice is taught early at the university level and to very
many students---most of whom do not go on to become mathematics majors.
Moreover, computer algebra systems have become widespread, including
good free ones, some of which are available for smartphones. Most
current computer algebra systems apply current textbook rules and
amplify the effects of fundamental ``minor'' errors such as the
error in continuity that we address in this article. So in practice,
the correction we present is important.

In order to promote the ideas of comprehensive antiderivatives
and Kahanian constants, we have developed
a Mathematica program that computes comprehensive
anti-derivatives and Kahanian forms. It can be used as an enhancement to Mathematica's
own \texttt{Integrate} command.
The program, as this paper, hopes to promote the analytical view common to several older calculus texts, such
as those by Apostol~\cite{Apostol57} or Courant and John~\cite{CourantJohn},
over the more common current practice
of making algebraic antidifferentiation the fundamental object of study.

The enhanced integration command is available on the website

\href{https://www.uwo.ca/apmaths/faculty/jeffrey/research/integration/index.html}{https://www.uwo.ca/apmaths/faculty/jeffrey/research/integration/index.html}

\bibliography{Kahanian}

\begin{thebibliography}{10}

\bibitem{Abramowitz1964}
Milton Abramowitz and Irene~A. Stegun.
\newblock {\em Handbook of Mathematical Functions: with Formulas, Graphs, and
  Mathematical Tables}.
\newblock Dover, 1964.

\bibitem{Essex}
Robert~A. Adams and Christopher Essex.
\newblock {\em Calculus, a complete course}.
\newblock Pearson, 8th edition, 2014.

\bibitem{Apostol57}
Tom~M. Apostol.
\newblock {\em Mathematical Analysis: A Modern Approach to Advanced Calculus}.
\newblock Addison-Wesley, 1957.

\bibitem{BenderOrszag}
Carl~M. Bender and Steven~A. Orszag.
\newblock {\em Advanced Mathematical Methods for Scientists and Engineers}.
\newblock McGraw-Hill, New York, 1978.

\bibitem{Bronstein2005}
Manuel Bronstein.
\newblock {\em Symbolic Integration {I}. {T}ranscendental Functions}.
\newblock Algorithms and Computation in Mathematics. Springer, 2nd edition,
  2005.

\bibitem{CorlessJeffrey:Simple}
R.~M. Corless and D.~J. Jeffrey.
\newblock Well... it isn't quite that simple.
\newblock {\em SIGSAM Bulletin}, 26(3):2--6, 1992.

\bibitem{corless2000according}
Robert~M. Corless, David~J. Jeffrey, Stephen~M. Watt, and James~H. Davenport.
\newblock ``{A}ccording to {Abramowitz} and {Stegun}'' or arccoth needn't be
  uncouth.
\newblock {\em ACM SIGSAM Bulletin: Communications in Computer Algebra},
  34(2):58--65, 2000.

\bibitem{CourantJohn}
Richard Courant and Fritz John.
\newblock {\em Introduction to Calculus and Analysis {I}}.
\newblock Springer, 1999.

\bibitem{Gradsh}
Izrail~Solomonovich Gradshteyn and Iosif~Moiseevich Ryzhik.
\newblock {\em Table of Integrals, Series, and Products}.
\newblock Academic Press, 6th edition, 2007.

\bibitem{Hawkins1977weier}
Thomas Hawkins.
\newblock {W}eierstrass and the theory of matrices.
\newblock {\em Archive for History of Exact Sciences}, 17(2):119--163, 1977.

\bibitem{JeffreyCorlessLies}
D.~J. Jeffrey and Robert~M. Corless.
\newblock Teaching linear algebra with and to computers.
\newblock In Wei-Chi Yang, editor, {\em Proceedings of ATCM 2001}, pages
  120--129, 2001.

\bibitem{JeffreyRich2010}
D.~J. Jeffrey and A.D. Rich.
\newblock Reducing expression size using rule-based integration.
\newblock In S.~Autexier, editor, {\em Intelligent Computer Mathematics},
  volume 6167 of {\em LNAI}, pages 234--246. Springer, 2010.

\bibitem{JeffreysJeffreys}
Harold Jeffreys and Bertha Swirles.
\newblock {\em Methods of Mathematical Physics}.
\newblock Cambridge University Press, 3rd edition, 1972.

\bibitem{Kahan59}
W.~M. Kahan.
\newblock Integral of $x^n$.
\newblock Personal communication, \textsl{circa} 1976.

\bibitem{MenckenAfflatus}
H.~L. Mencken.
\newblock The divine afflatus.
\newblock In {\em Prejudices. Second Series}, chapter~IV, pages 155--171.
  Alfred A. Knopf, New York, 1920.

\bibitem{PrudnikovBM}
Anatolii~Platonovich Prudnikov, Yuri~A. Brychkov, and Oleg~Igorevich Marichev.
\newblock {\em Integrals and Series}.
\newblock CRC Press, 1992.

\bibitem{Risch1970}
Robert~H. Risch.
\newblock The solution of the problem of integration in finite terms.
\newblock {\em Bull. Amer. Math. Soc.}, 76(3):605--608, 1970.

\bibitem{Stewart}
James Stewart.
\newblock {\em Calculus}.
\newblock Brooks Cole, 7th edition, 2012.

\bibitem{Timofeev1948}
Adrian~Fedorovich Timofeev.
\newblock {\em Integrirovanie funktsii (Integration of Functions)}.
\newblock Publishing House tehniko-teoreticheskoj literature, Leningrad, 1948.

\bibitem{CRCtables}
Daniel Zwillinger.
\newblock {\em Standard Mathematical Tables and Formulae}.
\newblock CRC Press, 32nd edition, 2011.

\end{thebibliography}
\bibliographystyle{plain}

\end{document}